\documentclass[a4paper,11pt,twoside]{amsart}
\usepackage[margin=1in]{geometry}
\usepackage{mathtools}
\usepackage{amsfonts}
\usepackage{amsthm}
\usepackage{amsmath}
\usepackage{amssymb}
\usepackage{enumerate}
\usepackage{hyperref}

\newcommand{\Z}{\mathbb{Z}}
\newcommand{\cyl}{\mathrm{cyl}}

\newtheorem{theorem}{Theorem}
\newtheorem{proposition}{Proposition}

\begin{document}
\title{A concise proof of cylindric Schur positivity}
\author{Alexander Dobner}
\begin{abstract}
Cylindric Schur functions are a family of symmetric functions that generalize skew Schur functions. We give a short proof that skew cylindric Schur functions expand positively in terms of non-skew cylindric Schur functions. In particular, we show that the expansion coefficients are fusion coefficients.
\end{abstract}
\maketitle
\vspace{-1em}

\section{Introduction}
The cylindric Schur functions are a family of symmetric functions which are generating functions for cylindric semistandard Young tableaux. They were introduced by Postnikov \cite{postnikov2005} in his study of the quantum cohomology of Grassmannians. They are also related to fusion coefficients in type A. (The connection between quantum cohomology of Grassmannians and fusion coefficients mirrors the connection between ordinary cohomology of Grassmannians and Littlewood-Richardson coefficients.)

McNamara \cite{mcnamara2006} conjectured that the skew cylindric Schur functions expand positively in terms of the non-skew cylindric Schur functions. This was first proved by Lee \cite{lee2019}, and a different proof was given by Korff and Palazzo \cite{korff2020p}. Both of these proofs are rather lengthy. In this note we give a very short proof of this result.

\begin{theorem}
Let $N$ and $L$ be positive integers. For any $(N,L)$-cylindric skew shape $\lambda/\mu$ we have the following expansion of the skew $(N,L)$-cylindric Schur function $s^\cyl_{\lambda/\mu}$ in terms of non-skew $(N,L)$-cylindric Schur functions:
\[s^\cyl_{\lambda/\mu} = \sum_{\nu} d_{\mu\nu}^{\lambda} s^\cyl_{\nu}.\]
The coefficients $d_{\mu\nu}^{\lambda}$ are the fusion coefficients of rank $N$ and level $L$. In particular, they are nonnegative integers.
\end{theorem}

This theorem naturally generalizes the classical fact that the skew Schur functions expand positively in terms of non-skew Schur functions. Indeed, when the level $L$ is sufficiently large (e.g. $L \geq \lambda_1$), the cylindric Schur function $s^\cyl_{\lambda/\mu}$ coincides with the ordinary skew Schur function $s_{\lambda/\mu}$, and the fusion coefficients $d_{\mu\nu}^{\lambda}$ in the sum coincide with the ordinary Littlewood-Richardson coefficients $c_{\mu\nu}^{\lambda}$.

\section{Definitions and proof}
\subsection{Partitions and cylindric tableaux}
A \emph{partition} is a weakly decreasing sequence of nonnegative integers $\lambda = (\lambda_1, \lambda_2, \ldots)$ with finitely many nonzero terms. The \emph{size} of $\lambda$ is $|\lambda| = \sum_i \lambda_i$. The \emph{length} of $\lambda$ is $\ell(\lambda) = \max\{i : \lambda_i > 0\}$. The \emph{empty partition} is $\emptyset = (0,0,\ldots)$.

A \emph{skew shape} is an ordered pair of partitions $(\mu,\lambda)$ such that $\mu_i \leq \lambda_i$ for all $i$. We write this ordered pair as $\lambda/\mu$. The \emph{size} of $\lambda/\mu$ is $|\lambda/\mu| = |\lambda| - |\mu|$. A \emph{horizontal strip} is a skew shape $\lambda/\mu$ such that $\lambda_1 \geq \mu_1 \geq \lambda_2 \geq \mu_2 \geq \cdots$. (Note: in terms of Young diagrams, $\lambda/\mu$ is a horizontal strip if it contains at most one box in each column.)

Fix a pair of positive integers $(N,L)$ throughout the rest of the paper. An \emph{$(N,L)$-partition} is a partition $\lambda$ such that $\ell(\lambda) \leq N$ and $\lambda_1-\lambda_N \leq L$. An \emph{$(N,L)$-horizontal strip} is a horizontal strip $\lambda/\mu$ such that $\ell(\lambda) \leq N$ and $\lambda_1-\mu_N \leq L$. Note that if $\lambda/\mu$ is an $(N,L)$-horizontal strip then $\lambda$ and $\mu$ are necessarily $(N,L)$-partitions and $|\lambda/\mu| \leq L$.

A skew shape $\lambda/\mu$ is \emph{$(N,L)$-cylindric} if both $\lambda$ and $\mu$ are $(N,L)$-partitions. If $\lambda/\mu$ is $(N,L)$-cylindric and $\alpha=(\alpha_1,\ldots,\alpha_r)$ is a vector of nonnegative integers, an \emph{$(N,L)$-cylindric semistandard tableau} of shape $\lambda/\mu$ and weight $\alpha$ is a sequence of partitions $\mu = \nu^{(0)}, \nu^{(1)}, \ldots, \nu^{(r)} = \lambda$ such that $\nu^{(i)}/\nu^{(i-1)}$ is an $(N,L)$-horizontal strip of size $\alpha_i$ for each $i$. Let $K^\cyl_{\lambda/\mu,\alpha}$ denote the number of $(N,L)$-cylindric semistandard tableaux of shape $\lambda/\mu$ and weight $\alpha$. This quantity is nonzero only if $\alpha_i \leq L$ for each $i$, and $|\lambda/\mu| = \alpha_1+\cdots+\alpha_r$. Let $K^\cyl_{\lambda,\alpha} = K^\cyl_{\lambda/\emptyset,\alpha}$.

\subsection{Cylindric Schur functions}
Let $\Lambda$ denote the ring of symmetric functions (in infinitely many variables, with coefficients in $\Z$). For any partition $\lambda$, let $s_{\lambda}, m_{\lambda}$, and $h_{\lambda}$ denote the Schur function, monomial symmetric function, and complete homogeneous symmetric function corresponding to $\lambda$. Let $h_k$ denote the complete homogeneous symmetric function of degree $k$.

Given an $(N,L)$-cylindric skew shape $\lambda/\mu$, the \emph{$(N,L)$-cylindric Schur function} $s^\cyl_{\lambda/\mu}$ is defined as
\[s^\cyl_{\lambda/\mu} = \sum_\alpha K^\cyl_{\lambda/\mu,\alpha} m_\alpha\]
where the sum is over all partitions $\alpha$. Let $s^\cyl_{\lambda} = s^\cyl_{\lambda/\emptyset}$.

\subsection{The fusion ring}

Let $\Lambda^{(N)}$ denote the ring of symmetric polynomials in $N$ variables with coefficients in $\Z$. Let $I^{(N,L)}$ denote the ideal of $\Lambda^{(N)}$ generated by the Schur polynomials $s_{\lambda}$ where $\lambda_1-\lambda_N = L+1$. Let $\Lambda^{(N,L)} = \Lambda^{(N)}/I^{(N,L)}$. This is the \emph{fusion ring} (see \cite{morse2012s}).

For any partition $\lambda$, let $S_\lambda$ denote the image of the Schur polynomial $s_\lambda$ in $\Lambda^{(N,L)}$. Similarly, let $H_\lambda$ (resp. $H_k$) denote the image of the complete homogeneous symmetric function $h_\lambda$ (resp. $h_k$) in $\Lambda^{(N,L)}$.

Goodman and Wenzl \cite{goodman1990w} proved various properties of the ring $\Lambda^{(N,L)}$. In particular, they showed that the following hold.
\begin{enumerate}[(i)]
    \item The collection $\{S_\lambda : \text{$\lambda$ is an $(N,L)$-partition}\}$ is a $\Z$-basis for $\Lambda^{(N,L)}$.
    \item The structure constants for this basis are nonnegative integers $d_{\mu\nu}^{\lambda}$. (These are the \emph{fusion coefficients} of rank $N$ and level $L$.)
    \item For any $(N,L)$-partition $\eta$ and any $0 \leq k \leq L$, we have $S_{\eta} H_k = \sum_{\rho} S_{\rho}$ where the sum is taken over all partitions $\rho$ such that $\rho/\eta$ is an $(N,L)$-horizontal strip of size $k$. (This is the \emph{fusion Pieri rule}.)
\end{enumerate}

By iteratively applying the fusion Pieri rule, we find that for any $(N,L)$-partition $\eta$ and any partition $\alpha$ satisfying $\alpha_1 \leq L$,
\begin{equation} \label{eq:fusion-pieri}
    S_{\eta} H_{\alpha} = \sum_{\rho} K^{\cyl}_{\rho/\eta, \alpha} S_{\rho}
\end{equation}
where the sum is taken over $(N,L)$-partitions $\rho$ such that $\rho/\eta$ is a skew shape.

\subsection{Proof of Theorem 1}

\begin{proposition}
For any $(N,L)$-partitions $\lambda,\mu$ and any partition $\alpha$ we have
\[
K^\cyl_{\lambda/\mu,\alpha} = \sum_{\nu} d^\lambda_{\mu\nu} K^\cyl_{\nu,\alpha}
\]
where the sum is taken over all $(N,L)$-partitions $\nu$.
\end{proposition}
\begin{proof}
If $\alpha_1 > L$ then both sides are zero, so we may assume $\alpha_1 \leq L$.
Applying \eqref{eq:fusion-pieri} in the $\eta=\emptyset$ case, we get $H_{\alpha} = \sum_{\nu} K^\cyl_{\nu,\alpha} S_{\nu}$. Multiplying the left and right sides by $S_{\mu}$, we get
\begin{align*}
S_\mu \cdot (\text{left side}) &= S_{\mu} H_{\alpha} = \sum_{\lambda} K^\cyl_{\lambda/\mu,\alpha} S_{\lambda} \\
S_\mu \cdot (\text{right side}) &= S_{\mu} \sum_{\nu} K^\cyl_{\nu,\alpha} S_{\nu} = \sum_{\nu} \sum_{\lambda}  K^\cyl_{\nu,\alpha} d^\lambda_{\mu\nu} S_{\lambda}.
\end{align*}
Since the $S_{\lambda}$ form a basis for $\Lambda^{(N,L)}$, we can equate coefficients to get the desired result.
\end{proof}

Proposition 1 immediately implies Theorem 1 by multiplying both sides of the equation by $m_{\alpha}$ and summing over $\alpha$.


\bibliographystyle{plain}
\bibliography{references}

\end{document}